\documentclass[11pt]{article}
\usepackage{amsmath}
\usepackage{amsfonts}
\usepackage{amssymb}
\usepackage{color}
\newtheorem{theorem}{Theorem}[section]
\newtheorem{corollary}{Corollary}[section]
\newtheorem{definition}{Definition}[section]

\newtheorem{lemma}{Lemma}[section]

\newtheorem{remark}{Remark}[section]
\newenvironment{proof}[1][Proof]{\noindent\textbf{#1.} }{\ \rule{0.5em}{0.5em}}

\def\ds{\displaystyle}
\begin{document}
\title{\textbf{\sc Existence results for integro-differential \\ equations
with reflection}}
\author{ {   \Large \bf \sc Mohsen Miraoui}$^{a,b}$ {\Large \bf \sc and  Du\v{s}an D. Repov\v{s}}$^{c,d,e}$\\
 $^a$ {\small \sc IPEIK, Kairouan University, Tunisia}\\
$^b${\small LR11ES53, FSS, Sfax University, Tunisia}\\
{\small {\it miraoui.mohsen@yahoo.fr}}\\
$^c$ {\small Faculty of Education, University of Ljubljana,  Slovenia}\\
$^d$ {\small Faculty of Mathematics and Physics, University of Ljubljana, Slovenia}\\
$^e$ {\small Institute of Mathematics, Physics and Mechanics, Ljubljana, Slovenia}\\
{\small
 {\it dusan.repovs@guest.arnes.si}}
}
\date{}
\maketitle
\begin{abstract} We prove several
  important results concerning  existence and
uniqueness of pseudo almost automorphic (paa)
solutions with measure for integro-differential equations with reflection. We use the properties of almost automorphic functions with measure and the Banach fixed point theorem, and we discuss two linear and nonlinear cases. We conclude with an example and some observations.
\vspace{5mm}
\vspace{3mm}
\end{abstract}
\section{Introduction}
Many authors have studied problems of existence of periodic,
   almost periodic and  automorphic solutions for different kinds of differential and integral equations
    (cf.
    Adivar and Koyuncuo\v{g}lu
     \cite{Ad},
    Baskakov {\sl et al.},
     \cite{ba},
    Bochner
     \cite{Boch},
    N'Gu\'er\'ekata
     \cite{g},
    and
    Papageorgiou {\sl et al.}
     \cite{D2,D1}).
    For example, the function
     $$t\rightarrow \sin t +\sin \sqrt{2} t$$
     is almost periodic but not periodic on
     $\mathbb{R}$, whereas the function
    $$t\rightarrow \sin\Big(\frac{1}{2+\cos t
+\cos \sqrt{2}t}\Big)$$
 is almost automorphic but not uniformly continuous,  hence not almost periodic
 on
 $\mathbb{R}$.

Recently, these research directions
have taken various generalizations (cf.
Ait Dads {\sl et al.}
\cite{A1, Ait,ait2020},
Ben-Salah {\sl et al.}
\cite{mounir},
Blot {\sl et al.}
\cite{Khalil2},
Ch\'{e}rif and Miraoui
\cite{ref4},
Diagana {\sl et al.}
\cite{D11},
Li
\cite{li},
Miraoui
\cite{mir3,mir2},
Miraoui {\sl et al.}
\cite{mir4},
Miraoui and Yaakobi
\cite{ref5},
and
Zhang
\cite{14}), as well as various applications (cf. e.g.
Kong and Nieto
\cite{KN},
and the references therein).

Let $\mu$
be
 positive
measure on $\mathbb{R}$ and
$X$
 a Banach space.
 A continuous function
$f: \mathbb{R} \mapsto X$
 is
said to be
{\it   measure paa}   (cf.
Ait Dads {\sl et al.}
 \cite{A1}
and
Papageorgiou {\sl et al.}
 \cite{D1}), if  $f$ can be written as a sum of an almost periodic function $g_1$ and an ergodic function
$\varphi_1$ satisfying
\begin{equation*}
\displaystyle \lim_{z \to\infty} \frac{1}{\mu([-z,z])} \int_{-z}^z
\|\varphi_1(y)\| d\mu(y)=0,
\end{equation*}
where
$$\displaystyle \mu([-z,z]) :=\int_{-z}^z d\mu(t).$$
Diagana
\cite{Diagana}
defined the network of weighted pseudo almost periodic functions, which generalizes the pseudo almost periodicity in
Gupta
\cite{8}.

Motivated by above mentioned work, we investigate in the present paper  measure paa solutions of differential equations involving reflection of the argument. This type of differential equations has applications in the study of stability of differential-difference equations, cf.  e.g.
Sharkovskii
\cite{sc},
and such equations show very interesting properties by themselves. Therefore several authors have worked on this category of equations.

Aftabizadeh {\sl et al.}
\cite{2},
Aftabizadeh and Wiener
 \cite{1},
and
Gupta
\cite{9}
 studied the existence of unique bounded solution of  equation
\begin{equation*}
u^{\prime }(y)=f(y,u(y),u(-y)),\;y\in
\mathbb{R}.
\end{equation*}
They proved that $u(y)$ is almost periodic by assuming the existence of bounded solution.
Piao
\cite{11,12} studied
the  following equations
\begin{equation}  \label{X2}
u^{\prime }(y) = au(y)+bu(-y)+g(y),\; b\neq 0, \;y\in\mathbb{R},
\end{equation}
and
\begin{equation}
u^{\prime }(y)=au(y)+bu(-y)+f(y,u(y),u(-y)),\;\;b\neq 0,\; y\in
\mathbb{R}. \label{X3}
\end{equation}
Xin and Piao
\cite{X}
obtained some results of weighted pseudo almost periodic  solutions for
equations
\eqref{X2} and \eqref{X3}.  
Recently, Miraoui \cite{miraoui2020} has studied the pseudo almost periodic (pap) solutions with two measures of equations \eqref{X2} and \eqref{X3}.

Throughout this paper, we shall assume the following hypothesis:\\

{($M_0$)}: There exists a continuous and strictly increasing function $\beta:\mathbb{R}\to\mathbb{R}$ such that for all $x\in\mathcal{AA}(\mathbb{R},\mathbb{R})$, we have $x \circ \beta\in \mathcal{AA}(\mathbb{R},\mathbb{R})$.\\

The key goal of our paper is to study equations which are more general than equations \eqref{X2} and \eqref{X3}, and are given by the following expression
\begin{eqnarray}\label{X10}
  u^{\prime }(y) &=& au(y)+bu(-y)+f(y,u(\beta (y)),u(\beta (-y)))\nonumber\\
  &+&\int_{y}^{+\infty
}K(s-y)h(s,u(\beta(s)),u(\beta(-s)))ds \\
  &+& \int_{-y}^{+\infty }K(s+y)h(s,u(s),u(-s))ds,\; y\in
\mathbb{R},\nonumber
\end{eqnarray}
where $a\in \mathbb{R}$, $b\in \mathbb{R}^*$, $f,h:\mathbb{R}^{3}\rightarrow \mathbb{R},$ and $K:\mathbb{R}^{+}\rightarrow
\mathbb{R}^{+}$ are continuous functions.

Let $\mathbb{X}$  be a  Banach space. We begin by defining the notion of a measure pseudo almost
automorphic function.

\begin{definition}(Bochner
\cite{Boch})
Let
  $ f \in \mathcal{C}(\mathbb{R},\mathbb{X})$.
  Then $f$ is said to be almost automorphic,
 $f\in \mathcal{AA}( \mathbb{R} , \mathbb{X})$, if for every real sequence $ (s_n),$ there
exists a subsequence $ (s_{n_k}) $,
such that
the following  limits
$$
\lim_{ n_k \rightarrow \infty} f ( t + s_{n_k} )=f(t)   \ \
\mbox{and}  \ \
\lim_{n_k \rightarrow  \infty} g ( t - s_{n_k} ) = f (t ) ,
$$
exist for every  $ t \in \mathbb{R}  $.
\end{definition}

\begin{definition}(Blot et al.
\cite{Khalil2})
Let  $\mathcal{B}$ is the Lebesque $\sigma$-field of
$\mathbb{R}$
and $\mu$ a positive measure on $\mathcal{B}$. Then
$\mu \in \mathcal{M}$ if the following conditions are satisfied
 \begin{itemize}
\item $\mu([a,b])<\infty$, for all $a\leq b \in \mathbb{R}$; and
\item $\mu(\mathbb{R})=+\infty$.
 \end{itemize}
\end{definition}

In this paper we shall be working with a positive measure satisfying the following two important hypotheses:\\

{($M_1$)} For every $\tau\in \mathbb{R}$, there exist $\beta >0$
and a bounded interval $I$  such that
\begin{equation*}
\mu(\{a+\tau: \  a \in
A\})\leq \beta \mu(A), \  \mbox{whenever} \
A \in \mathcal{B} \ \mbox{satisfies} \ A\cap I=\emptyset .
\end{equation*}

{($M_2$)} There exist $m,n >0$  such that for all $A\in \mathcal{B}$,
\begin{equation*}
\mu(-A) \leq m+n\mu(A).
\end{equation*}

\begin{definition}(Diagana et al.
\cite{D11})
Suppose that $\mu \in \mathcal{M}$.
Then $f \in \mathcal{BC}(\mathbb{R}, \mathbb{X})$ is said to be $\mu$-ergodic, $f\in \mathcal{E}(\mathbb{R
},\mathbb{X},\mu)$, if the following condition is satisfied:
\begin{equation*}
\displaystyle \lim_{z \to \infty} \frac{1}{\mu([-z,z])}
\int_{[-z,z]} \|f(y)\|d\mu(y)=0.
\end{equation*}
\end{definition}
\begin{definition}(Diagana et al.
\cite{D11})
Suppose that $\mu \in \mathcal{M}$. Then $f\in \mathcal{C}(\mathbb{R},\mathbb{X})$ is said to be  $\mu$-paa, $f\in PAA(\mathbb{R}
,\mathbb{X},\mu)$, if
\begin{equation*}
f=g+h,
\end{equation*}
where $g\in AA(\mathbb{R},\mathbb{X})$  and the function $h$ is $\mu$-ergodic.
\end{definition}
 In the sequel, we shall also need the following hypotheses
\begin{itemize}
\item[{($h_0$)}] There exists a continuous, strictly increasing function
$\lambda:\mathbb{R}\to\mathbb{R}^+$ such that
    $d\mu_{\beta}(t)\leq \lambda(t)d \mu(t)$, where $\mu \in \mathcal{M}$, $\mu_{\beta}(O)=\mu(\beta^{-1}(O))$, for all $O\in \mathbb{B}(\mathbb{R})$ and
    $$\limsup \frac{\mu[-T(r),T(r)]}{\mu[-r,r]}S(T(r))<+\infty,$$ where $T(r)=|\beta(r)|+|\beta(-r)|$ and $S(T(r))=\sup_{t\in [-T(r),T(r)]}\lambda(t)$.
    \item[{($h_1$)}] Given $\lambda
:=\sqrt{a^{2}-b^{2}}$, where $a>b$, the following holds
\begin{equation*}
P_1(\lambda ,\mu ):=\sup_{z>0}\Big\{\int_{-z}^{z}\exp (-\lambda (t+z))d\mu (t)
\Big\}<\infty,
\end{equation*}
\begin{equation*}
P_2(\lambda ,\mu ):=\sup_{z>0}\Big\{\int_{-z}^{z}\exp (-\lambda (-t+z))d\mu (t)
\Big\}<\infty .
\end{equation*}
\item[{($h_2$)}]  There exists $L_{f}> 0$, such that $f:\mathbb{R}
\times \mathbb{R}^2\rightarrow \mathbb{R}$ satisfies the Lipschitz
condition
\begin{equation*}
|f(t,x_1,y_1)-f(t,x_2,y_2)|\leq L_f\Big(|x_1-x_2|+|y_1-y_2|\Big), \ \hbox{for all} \
(x_1,y_1), (x_2,y_2)\in \mathbb{R}^2.
\end{equation*}
 \item [{($h_3$)}] There exists $L_{h}> 0$ such that
\begin{equation*}
| h( t, u_1, u_2) - h( t, v_1 , v_2 ) | < L_h (| u_1 - v_1 | + |
u_2 - v_2| ), \
\hbox{for all} \; u_1, u_2, v_1
, v_2 \in \mathbb{R}.
\end{equation*}
\item [{($h_4$)}] There exists $K : \mathbb{R}^+\rightarrow \mathbb{R}^+ $ such that
\begin{equation*}
c:=\int_{0}^{+ \infty} K(y) dy <\infty.
\end{equation*}
\end{itemize}

Our first main result of the paper treats the case when $L_f$ and $L_h$ are constant.
\begin{theorem}\label{th1}
Suppose that $f,h  \in \mathcal{PAA}(\mathbb{R},\mathbb{R},\mu)$ and
that hypotheses {($h_0$)--($h_4$)} and {($M_0$)--($M_2$)}
hold. Then equation \eqref{X10} has a unique $\mu$-paa solution if and only if
\begin{equation*}
 \displaystyle\frac{|\lambda-a|+|\lambda+a|+2|b|}{\lambda^2}(L_f+2cL_h)
< 1.
\end{equation*}
\end{theorem}
For the second main result of this paper we shall need the following hypotheses for the case when $L_f$ and $L_h$ are not constant.
\begin{itemize}
\item[{($h'_2$)}] $\mu \in \mathcal{M}$ and $f:\mathbb{R}
\times \mathbb{R}^2\rightarrow \mathbb{R}$ satisfy
\begin{equation*}
|f(t,x_1,y_1)-f(t,x_2,y_2)|\leq L_f(t)\Big(|x_1-x_2|+|y_1-y_2|\Big), \
\hbox{for all} \
(x_1,y_1),(x_2,y_2)\in \mathbb{R}^2,
\end{equation*}
where
$p>1, L_f\in \mathcal{L}^p(\mathbb{R},\mathbb{R},dx)\cap \mathcal{L}^p(\mathbb{R},\mathbb{R},d\mu),
\ \mbox{and} \
\frac{1}{p}+\frac{1}{q}=1.$
\item[{($h'_3$)}] $\mu \in \mathcal{M}$ and $h:\mathbb{R}
\times \mathbb{R}^2\rightarrow \mathbb{R}$ satisfy
\begin{equation*}
|h(t,x_1,y_1)-h(t,x_2,y_2)|\leq L_h(t)\Big(|x_1-x_2|+|y_1-y_2|\Big),\
\hbox{for all} \
(x_1,y_1), (x_2,y_2)\in \mathbb{R}^2,
\end{equation*}
 where
$p>1, L_h\in \mathcal{L}^p(\mathbb{R},\mathbb{R},dx)\cap \mathcal{L}^p(\mathbb{R},\mathbb{R},d\mu)
\ \mbox{and} \
\frac{1}{p}+\frac{1}{q}=1.$
\item[{($h'_4$)}] There exists $K:\mathbb{R}^{+}\rightarrow \mathbb{R}^{+}$, such that
$$\ds\int_{0}^{+\infty }(K(y))^{\tau}dy <+\infty,
 \ \hbox{for all} \ \tau>1.$$
\end{itemize}
\begin{theorem}
\label{final2} Suppose that $f,h \in PAA(\mathbb{R}\times\mathbb{R}
^{2},\mathbb{R},\mu)$ and that  hypotheses
{($h_0$)-($h_1$),  ($h'_2$)--($h'_4$)
and ($M_0$)--($M_2$)} hold. Then equation \eqref{X10} has a unique
$\mu$-paa solution
 if and only if
$$\|L_f\|_{\mathcal{L}^p(\mathbb{R},\mathbb{R},dx)}+2(\int_0^{+\infty}(K(y))^q)^{\frac{1}{q}}\|L_h\|_{\mathcal{L}^p(\mathbb{R},\mathbb{R},dx)}<\displaystyle
\frac{\lambda(q\lambda)^{\frac{1}{q}} }{|\lambda-a|+|\lambda+a|+2|b| }.$$
\end{theorem}
We conclude the introduction by description of the structure of the paper.
In Section 2,   we  collect some basic results needed for the proofs of the main results of this paper.  In section 3, we  prove both main results (Theorems 1.1 and 1.2).  In Section 4,  we give an application of the measure paa, in connection with integro-differential equations with reflection and delay. In  Section 5 we discuss the results and their applications.

\section{\textbf{Preliminaries}}
\begin{theorem}(Diagana et al.
\cite{D11})
\label{theorem2.8}
Suppose that $\mu \in \mathcal{M}$ satisfies hypothesis {($M_1$)}. Then  $PAA(\mathbb{R}
,X,\mu)$ is translation invariant and   $(PAA(\mathbb{R},\mathbb{X},\mu),\|.\|_{\infty})$ is
a Banach space.
\end{theorem}
\begin{lemma}(Miraoui
\cite{miraoui2020})
\label{opp}
Suppose that $g\in PAA(\mathbb{R},\mathbb{X},\mu)$ and that hypothesis ({$M_2$}) holds.
Then $$[t\rightarrow g(-t)]\in PAA(\mathbb{R}
,\mathbb{X},\mu).$$
\end{lemma}
\begin{lemma} (Miraoui \cite{miraoui2020})\label{lemma2}
If $\mu \in \mathcal{M}$ satisfies hypothesis {($M_1$)}, then for all $ p\geq 1,$
$$L^p(\mathbb{R},\mathbb{X},d\mu)\subset \mathcal{E}(\mathbb{R},\mathbb{X},\mu).$$
\end{lemma}
\begin{lemma}(Ben Salah et al.\cite{mounir})
\label{comp0} Suppose that
hypotheses
 {($h_0$)} and {($M_0$)} hold. If   $v\in
\mathcal{PAA}(\mathbb{R},\mathbb{R},\mu),$ then $[t\mapsto
v(\beta (t))]\in \mathcal{PAA}(\mathbb{R},\mathbb{R},\mu ).$
\end{lemma}
\begin{lemma}
\label{comp2} Suppose that
hypotheses
 {($h_0$),($h_2$)} and {($M_0$)}--{($M_2$)} hold. If $f\in \mathcal{PAA}(\mathbb{R}^{3},\mathbb{R},\mu), $  and  $v\in
\mathcal{PAA}(\mathbb{R},\mathbb{R},\mu),$ then $[t\mapsto
f(t,v(\beta (t)),v(\beta (-t)))]\in \mathcal{PAA}(\mathbb{R},\mathbb{R},\mu ).$
\end{lemma}
\begin{proof}Let
$f\in \mathcal{PAA}(\mathbb{R}^{3},\mathbb{R},\mu
) $. Then  $f$ can be written as $f=h+\varphi ,$ where $h\in\mathcal{AAU}(\mathbb{R}^{3},\mathbb{R}), \varphi\in\mathcal{EU}(\mathbb{R}^{3},\mathbb{R},\mu )$ (see \cite{A1}). We set $V(t)=v(\beta (t)$, for all $t\in \mathbb{R}$. By Lemma \ref{comp0}, we can conclude that $V\in \mathcal{PAA}(\mathbb{R},
\mathbb{R},\mu ),$
hence $V=V_{1}+V_{2},$ where $V_{1}\in \mathcal{AA}(\mathbb{R},\mathbb{R}),V_{2}\in \mathcal{E}(\mathbb{R},\mathbb{R},\mu )$,
and so we have
\begin{eqnarray*}
f(t,V(t),V(-t))&=&\varphi_1(t,V_{1}(t),V_{1}(-t))+
f(t,V(t),V(-t))\\
&-&f(t,V_{1}(t),V_{1}(-t))+\varphi_2
(t,V_{1}(t),V_{1}(-t)).
\end{eqnarray*}
On the one hand, we shall prove that $[t\rightarrow \varphi_1(t,V_{1}(t),V_{1}(-t))]\in \mathcal{AAU}(\mathbb{R}^{3},%
\mathbb{R})$. Let $H(t)=\varphi_1(t,V_{1}(t),V_{1}(-t)).$  If $\{s_{n}\}$ is a sequence of real numbers, then we can extract a subsequence $\{\tau _{n}\}$
of $\{s_{n}\}$ such that\\

(1) $\displaystyle\lim_{n\rightarrow \infty }\varphi_1(t+\tau _{n},v,u)=\phi
(t,v,u), $ for all $t,v,u\in \mathbb{R};$ 

(2) $\displaystyle\lim_{n\rightarrow \infty }\phi (t-\tau
_{n},v,u)=\varphi_1(t,v,u), $ for all $t,v,u\in \mathbb{R};$ 

(3) $\displaystyle\lim_{n\rightarrow \infty }V_{1}(t+\tau
_{n},v,u)=U_{1}(t,v,u),$ for all $t,v,u\in \mathbb{R};$ 

(4) $\displaystyle\lim_{n\rightarrow \infty }U_{1}(t-\tau
_{n},v,u)=V_{1}(t,v,u),$ for all $t,v,u\in \mathbb{R}.$\\

\noindent
If  $\Phi (t):\mathbb{R}\longrightarrow \mathbb{R}$ by $\Phi
(t)=\phi (t,V_{1}(t),U_{1}(t)),$ then we can show that
\begin{equation*}
\lim_{n\rightarrow \infty }H(t+\tau _{n})=\Phi (t);\;\lim_{n\rightarrow
\infty }\Phi (t-\tau _{n})=H(t),\ \hbox{for all} \ t \in \mathbb{R}
\end{equation*}
and we get
\begin{eqnarray*}
\Vert H(t+\tau _{n})-\Phi (t)\Vert &\leq &\Vert \varphi_1(t+\tau _{n},V_{1}(t+\tau
_{n}),V_{1}(-t+\tau _{n}))-\varphi_1(t+\tau _{n},U_{1}(t),U_{1}(-t))\Vert \\
&+&\Vert \varphi_1(t+\tau _{n},U-1(t),U_{1}(-t))-\phi (t,U_{1}(t),U_{1}(-t))\Vert .
\end{eqnarray*}

Since $V_{1}(t)$ is almost automorphic, it follows that $V_{1}(t),$ and $U_{1}(t)$ are
bounded. Therefore there exists a bounded subset $K\subset \mathbb{R}.$ From
(3) and  (\textbf{$h_2$}), we see that $\varphi_1(t,V_{1}(t),V_{1}(-t))$ are
uniformly continuous on every bounded subset $K\subset \mathbb{R},$ hence
\begin{equation*}
\lim_{n\rightarrow \infty }\Vert \varphi_1(t+\tau _{n},V_{1}(t+\tau
_{n}),V_{1}(-t+\tau _{n}))-\varphi_1(t+\tau _{n},U_{1}(t),U_{1}(-t))\Vert =0
\end{equation*}
therefore
\begin{equation*}
\lim_{n\rightarrow \infty }\Phi (t-\tau _{n})=H(t),\;\text{for\ all}\;t\;\in
\;\mathbb{R}.
\end{equation*}
This proves that $H$ is an almost automorphic function.
On the other hand, we shall
show that $[t\rightarrow
f(t,V(t),V(-t))-f(t,V_{1}(t),V_{1}(-t))]\in \mathcal{E}(\mathbb{R},\mathbb{R}
,\mu ).$

We consider now the following function $\Phi (t)=f(t,V(t),V(-t))-f(t,V_{1}(t),V_{1}(-t)).$ Clearly, $\Phi (t)\in
\mathcal{BC}(\mathbb{R},\mathbb{R}).$
Since
\begin{equation*}
\Vert f(t,u_{1},u_{2})-f(t,v_{1},v_{2})\Vert \leq L_{f}(\Vert
u_{1}-v_{1}\Vert +\Vert u_{2}-v_{2}\Vert),
\end{equation*}
we have
\begin{eqnarray*}
\frac{1}{\mu ([-r,r])}\int_{-r}^{r}\Vert \Phi (t)\Vert d\mu (t) &=&\frac{1}{
\mu ([-r,r])}\int_{-r}^{r}\Vert f(t,V(t),V(-t))-f(t,V_{1}(t),V_{1}(-t))\Vert
d\mu (t) \\
&\leq &\frac{1}{\mu ([-r,r])}\int_{-r}^{r}L_{f}^{1}\Vert V(t)-V_{1}(t)\Vert
+L_{f}^{2}\Vert V(-t)-V_{1}(-t)\Vert d\mu (t) \\
&\leq &\frac{L_{f}}{\mu ([-r,r])}\int_{-r}^{r}\Vert V_{2}(t)\Vert d\mu
(t)+\frac{L_{f}}{\mu ([-r,r])}\int_{-r}^{r}\Vert V_{2}(-t)\Vert d\mu (t),
\end{eqnarray*}
so by  Lemma \ref{opp},
\begin{equation*}
\displaystyle\lim_{r\rightarrow \infty }\frac{1}{\mu ([-r,r])}\displaystyle
\int_{-r}^{r}\Vert \Phi (t)\Vert d\mu (t)=0.
\end{equation*}
Therefore $[t\rightarrow f(t,V(t),V(-t))=f(t,v(\beta (t)),v(\beta (-t)))]\in \mathcal{PAA}(\mathbb{
R},\mathbb{R},\mu ).$
\end{proof}
\begin{lemma}
\label{safa2} Suppose that hypotheses {($h_0$),($h_2$), ($h_4$)} and ({$M_0$})--({$M_2$}) hold.
Then for every $h\in
\mathcal{PAA}(\mathbb{R}^3,\mathbb{R},\mu ), v\in
\mathcal{PAA}(\mathbb{R},\mathbb{R},\mu ),$
$$[t\mapsto \displaystyle\int_{t}^{+\infty }K(s-t)h(s,v(\beta (s)),v(\beta (-s)))ds]\in
\mathcal{PAA}(\mathbb{R},\mathbb{R},\mu).$$
\end{lemma}
\begin{proof}
By Lemma \ref{comp2}, we know that $[t\longmapsto h(t,v(\beta (t)),v(\beta (-t)))]\in
\mathcal{PAA}(\mathbb{R},\mathbb{R},\mu),$ so $$
 h(t,v(\beta (t)),v(\beta (-t)))=h_{1}(t)+h_{2}(t),$$ where $h_{1}\in \mathcal{AA}(\mathbb{R},
\mathbb{R})$ and $h_{2}\in \mathcal{E}(\mathbb{R},\mathbb{R},\mu ).$ Set $$\displaystyle\Theta (t)=\int_{t}^{+\infty }K(s-t)h(s,v(\beta (s)),v(\beta (-s)))ds.$$
Then
\begin{eqnarray*}
\Theta (t) &=&\int_{t}^{+\infty }K(s-t)h_{1}(s)ds+\int_{t}^{+\infty }K(s-t)h_{2}(s)ds =\theta _{1}(t)+\theta _{2}(t),
\end{eqnarray*}
where $$\displaystyle\theta _{1}(t)=\int_{t}^{+\infty }K(s-t)h_{1}(s)ds
 \ \hbox{and} \
\displaystyle\theta _{2}(t)=\int_{t}^{+\infty }K(s-t)h_{2}(s)ds.$$\\
Since $u_{1}\in \mathcal{AA}(\mathbb{R},\mathbb{R}),$
it follows that  for every
sequence ${(\tau _{n}^{\prime })}_{n\in \mathbb{N}}$ there exists a
subsequence ($\tau _{n})$ such that
\begin{equation}
h_{1}(t)=\lim_{n\rightarrow \infty }u_{1}(t+\tau _{n})  \label{aa}
\end{equation}
is well-defined for each $t\in \mathbb{R}$ and
\begin{equation}
\lim_{n\rightarrow \infty }h_{1}(t-\tau _{n})=u_{1}(t), \
\hbox{for each} \
t\in \mathbb{R}.
 \label{aaa}
\end{equation}
 Let $M(t)=\displaystyle\int_{t}^{+\infty }K(s-t)u_{1}(s)ds$. Then
\begin{eqnarray*}
|\theta _{1}(t)-M(t+s_{n})| &=&|\int_{t}^{+\infty
}K(s-t)h_{1}(s)ds-\int_{t+s_{n}}^{+\infty }K(s-t-s_{n})u_{1}(s)ds| \\
&=&|\int_{t}^{+\infty }K(s-t)(h_{1}(s)-u_{1}(s+s_{n}))ds|.
\end{eqnarray*}
Using Eq. \eqref{aa}, hypotheses {($h_4$)} and the LDC Theorem, it follows that
\begin{equation*}
\Vert \int_{t}^{+\infty }K(s-t)(h_{1}(s)-u_{1}(s+s_{n}))ds\Vert
\longrightarrow 0,\;\text{as}\;n\rightarrow \infty ,\;t\in\mathbb{R}.
\end{equation*}
Therefore
$
\theta _{1}(t)=\lim_{n\rightarrow \infty }M(t+\tau _{n}),\;\hbox{for all} \ t\in\mathbb{R}.
$
Using the same argument, we also obtain
$\lim_{n\rightarrow \infty }h_{1}(t-\tau _{n})=u_{1}(t).$
Therefore, $\theta _{1}\in AA(\mathbb{R},\mathbb{R}).$

 To prove that $\Theta (t)\ PAA(\mathbb{R},\mathbb{R},\mu),$ we need to show that $\theta _{2}\in \mathcal{E}(\mathbb{R},\mathbb{R},\mu)$. We know that
\begin{eqnarray*}
\lim_{r\rightarrow +\infty }\dfrac{1}{\mu \lbrack -r,r]}\int_{-r}^{r}\Vert
\theta _{2}(t)\Vert d\mu (t) &=&\lim_{r\rightarrow +\infty }\dfrac{1}{\mu
\lbrack -r,r]}\int_{-r}^{r}\int_{t}^{+\infty }\Vert K(s-t)h_{2}(s)ds\Vert
d\mu (t) \\
&\leq &\lim_{r\rightarrow +\infty }\dfrac{1}{\mu \lbrack -r,r]}
\int_{-r}^{r}\int_{t}^{+\infty }\Vert K(s-t)\Vert \Vert h_{2}(s)\Vert dsd\mu
(t) \\
&\leq &\lim_{r\rightarrow +\infty }\dfrac{1}{\mu \lbrack -r,r]}
\int_{-r}^{r}\int_{0}^{+\infty }\Vert K(y)\Vert \Vert h_{2}(y+t)\Vert dyd\mu
(t) \\
&= &\lim_{r\rightarrow +\infty }\int_{0}^{+\infty }\dfrac{K(y)}{\mu \lbrack -r,r]}
\int_{-r}^{r}\Vert h_{2}(y+t)\Vert d\mu (t)dy.
\end{eqnarray*}

By  the LDC Theorem and Theorem \ref{theorem2.8}, we have
\begin{equation*}
\displaystyle\lim_{r\rightarrow +\infty }\dfrac{1}{\mu \lbrack -r,r]}
\int_{-r}^{r}\Vert \theta _{2}(t)\Vert d\mu (t)\leq \int_{0}^{+\infty }K(y)\lim_{r\rightarrow +\infty }\dfrac{1}{\mu \lbrack
-r,r]}\int_{-r}^{r}\Vert h_{2}(y+t)\Vert d\mu (t)dy=0.
\end{equation*}
It follows that $[t\mapsto \displaystyle\int_{t}^{+\infty }K(s-t)h(s,v(\beta (s)),v(\beta (-s)))ds]\in
\mathcal{PAA}(\mathbb{R},\mathbb{R},\mu).$
\end{proof}
\begin{remark}
We have shown that
\begin{equation}\label{2019}
\Big{[}t \rightarrow \displaystyle \int_{t}^{+\infty}
K(s-t)h(s,v(\beta (s)),v(\beta (-s)))ds \Big{]} \in \mathcal{PAA}(\mathbb{R},\mathbb{R},\mu).
\end{equation}
From ({$M_2$}) and equation \eqref{2019} we can also obtain
\begin{equation*}
\Big{[}t \rightarrow \displaystyle \int_{-t}^{+\infty}
K(s+t)h(s,v(\beta (s)),v(\beta (-s)))ds \Big{]} \in \mathcal{PAA}(\mathbb{R},\mathbb{R},\mu).
\end{equation*}
\end{remark}
\begin{lemma}(Ben Salah et al.~\cite{mounir})\label{comp3}
Let $\mu\in \mathcal{M}$, $g \in PAA(\mathbb{R},\mathbb{R}^2,\mu)$,
 $h \in PAAU(\mathbb{R}
\times \mathbb{R}^2,\mathbb{R},\mu)$, and
suppose that hypotheses
  {($M_1$)}
  and
   {($h'_3$)} hold. Then $[t \longmapsto h(t,g(t))]\in PAA(\mathbb{R},\mathbb{R},\mu).$
\end{lemma}

\section{Proofs of Main Results}
\subsection{Proof of Theorem 1.1}
\begin{proof}
By
Aftabizadeh and Wiener
\cite{1}, for any
$f,h\in PAA(\mathbb{R},\mathbb{R},\mu),$ a particular solution of
equation \eqref{X2} is as follows
\begin{align}\label{g}
\Gamma x(t)=& -\frac{1}{2\lambda }\Big[\exp (\lambda t)\int_{t}^{\infty
}\exp
(-\lambda y)\Big((\lambda -a)f(y,x(y),x(-y))+bf(-y,x(-y),x(y))\Big)dy\Big]\nonumber \\
& +\frac{1}{2\lambda }\Big[\exp (-\lambda t)\int_{-\infty }^{t}\exp
(\lambda y)\Big((\lambda +a)f(y,x(y),x(-y))-bf(-y,x(-y),x(y))\Big)dy\Big]\nonumber\\
&- \frac{1}{2\lambda }\Big[\exp (\lambda t)\int_{t}^{\infty
}\exp
(-\lambda y)\Big((\lambda -a)g(y)+bg(-y)\Big)dy\Big]\nonumber \\
& +\frac{1}{2\lambda }\Big[\exp (-\lambda t)\int_{-\infty }^{t}\exp
(\lambda y)\Big((\lambda +a)g(y)-bg(-y)\Big)dy\Big],
\end{align}

where $$g(y)=\int_{y}^{+\infty
}K(s-y)h(s,u(\beta (s)),u(\beta(-s)))ds + \int_{-y}^{+\infty }K(s+y)h(s,u(\beta(s)),u(\beta(-s)))ds.$$
According to Lemmas \ref{opp}, \ref{comp2}, and \ref{safa2}, we can
conclude
\begin{equation*}
\lbrack y\mapsto \int_{y}^{+\infty }K(s-y)h(s,u(\beta (s)),u(\beta(-s)))ds]\in \mathcal{PAA}
(\mathbb{R},\mathbb{R},\mu ).
\end{equation*}
Also, by Lemma \ref{opp}, we have
\begin{equation*}
\lbrack y\mapsto \int_{-y}^{+\infty }K(s+y)h(s,u(\beta (s)),u(\beta(-s)))ds]\in \mathcal{PAA}
(\mathbb{R},\mathbb{R},\mu ).
\end{equation*}
Therefore $g \in \mathcal{PAA}(\mathbb{R},\mathbb{R},\mu )$. (We can also use the parity of  $g$ to see that $g\in\mathcal{PAA}(\mathbb{R},\mathbb{R},\mu )$.)

 So, using lemmas from Section 2, we can deduce that $\Gamma$ is a
mapping of $\mathcal{PAA}(\mathbb{R},\mathbb{R},\mu )$ into itself.
Set
\begin{eqnarray}\label{F} 
F(t,v(\beta(t)),v(\beta(-t)))
\nonumber &=&
f(t,v(\beta(t)),v(\beta(-t)))\\
 &+&
\int_{y}^{+\infty
}K(s-y)h(s,u(\beta (s)),u(\beta(-s)))ds\\
\nonumber &+& \int_{-y}^{+\infty }K(s+y)h(s,u(\beta(s)),u(\beta(-s)))ds.
\end{eqnarray}
It remains to show that
$\Gamma:\mathcal{PAA}(\mathbb{R},\mathbb{R},\mu )\to\mathcal{PAA}(\mathbb{R},\mathbb{R},\mu )$
 is a strict contraction.
 
Since by hypothesis {($M_0$)},
$\beta:\mathbb{R}\to\mathbb{R}$ is bijective,
it follows that for all $u,v\in \mathcal{PAA}(\mathbb{R},\mathbb{R},\mu )$,
the following holds
\vfill\eject
\begin{eqnarray*}
&&|F(t,v(\beta(t)),v(\beta(-t)))-F(t,u(\beta(t)),u(\beta(-t)))|\\
&=&|f(t,v(\beta(t)),v(\beta(-t)))-f(t,u(\beta(t)),u(\beta(-t)))| \\
&&+\int_{t}^{+\infty }K(s-t)\Big{(}h(s,v(\beta(s)),v(\beta(-s)))-h(s,u(\beta(s)),u(\beta(-s)))\Big{)}ds
\\
&&+\int_{-t}^{+\infty }K(t+s)(h(s,v(\beta(s)),v(\beta(-s)))-h(s,u(\beta(s)),u(\beta(-s))))ds \\
&\leq &|f(t,v(\beta(t)),v(\beta(-t)))-f(t,u(\beta(t)),u(\beta(-t)))| \\
&&+\int_{0}^{+\infty }K(s)\Big{(}
h((s+t),v(\beta(s+t)),v(\beta(-(s+t))))-h((s+t,u(\beta(s+t)),u(\beta(-(s+t))))\Big{)}ds \\
&&+\int_{0}^{+\infty }K(s)\Big{(}
h(s-t,v(\beta(s-t)),v(\beta(-(s-t))))-h(s-t,u(\beta(s-t)),u(\beta(-(s-t))))\Big{)}ds \\
&\leq &2(L_{f}+2cL_{h})\Vert v-u\Vert _{\infty },
\end{eqnarray*}
therefore
\begin{eqnarray*}
|\Gamma v(t)-\Gamma u(t)| &\leq &  \displaystyle\frac{|\lambda-a|+|\lambda+a|+2|b|}{\lambda^2}(L_{f}+2cL_{h})\Vert v-u\Vert _{\infty }.
\end{eqnarray*}
Since
\begin{equation*}
 \displaystyle\frac{|\lambda-a|+|\lambda+a|+2|b|}{\lambda^2}(L_f+2cL_h)
< 1,
\end{equation*}
 it follows that $\Gamma :\mathcal{PAA}(\mathbb{R},\mathbb{R},\mu
)\longrightarrow \mathcal{PAA}(\mathbb{R},\mathbb{R},\mu )$ is indeed a
 strict contraction. Therefore $\Gamma$ has a unique fixed point in  $\mathcal{PAA}(
\mathbb{R},\mathbb{R},\mu )$ and equation \eqref{X10} has a unique measure paa solution.
\end{proof}
\subsection{Proof of Theorem 1.2}
\begin{proof}
We consider the function $\Gamma$ defined in system \eqref{g}.
Using lemmas from Section 2 and paying attention to coefficients $L_f$ and $L_h$ which are not constants, we can deduce that $\Gamma$ is a mapping of $\mathcal{PAA}(\mathbb{R},\mathbb{R},\mu )$ into itself. It remains to show that $\Gamma$ is a strict contraction. Indeed, knowing that  $F$ is given by \eqref{F}, we have
\begin{eqnarray*}
&&|F(t,v(\beta(t)),v(\beta(-t)))-F(t,u(\beta(t)),u(\beta(-t)))| \\
&\leq &|f(t,v(\beta(t)),v(\beta(-t)))-f(t,u(\beta(t)),u(\beta(-t)))| \\
&&+\int_{0}^{+\infty }K(s)\Big{(}
h((s+t),v(\beta(s+t)),v(\beta(-(s+t))))-h((s+t,u(\beta(s+t)),u(\beta(-(s+t))))\Big{)}ds \\
&&+\int_{0}^{+\infty }K(s)\Big{(}
h(s-t,v(\beta(s-t)),v(\beta(-(s-t))))-h(s-t,u(\beta(s-t)),u(\beta(-(s-t))))\Big{)}ds \\
&\leq &\Big[2L_f(t)+4\Big(\int_0^{+
\infty}(K(y))^q dy\Big)^{\frac{1}{q}}\|L_h\|_{\mathcal{L}^p(\mathbb{R},\mathbb{R},dx)}\Big]\Vert v-u\Vert _{\infty },
\end{eqnarray*}
where $u,v\in \mathcal{PAA}(\mathbb{R},\mathbb{R},\mu )$, hence
\begin{eqnarray*}
|\Gamma v(t)-\Gamma u(t)| &\leq &  \displaystyle\frac{|\lambda-a|+|\lambda+a|+2|b|}{\lambda(q\lambda)^{\frac{1}{q}}}\Big[\|L_f\|_{\mathcal{L}^p(\mathbb{R},\mathbb{R},dx)}\\
&+&2(\int_0^{+\infty}(K(y))^qdy)^{\frac{1}{q}}\|L_h\|_{\mathcal{L}^p(\mathbb{R},\mathbb{R},dx)}\Big]\Vert v-u\Vert _{\infty }.
\end{eqnarray*}
Since
\begin{equation*}
\displaystyle\frac{|\lambda-a|+|\lambda+a|+2|b|}{\lambda(q\lambda)^{\frac{1}{q}}}\Big[\|L_f\|_{\mathcal{L}^p(\mathbb{R},\mathbb{R},dx)}
+2\Big(\int_0^{+\infty}(K(y))^q dy\Big)^{\frac{1}{q}}\|L_h\|_{\mathcal{L}^p(\mathbb{R},\mathbb{R},dx)}\Big]
< 1,
\end{equation*}
  the operator $\Gamma :\mathcal{PAA}(\mathbb{R},\mathbb{R},\mu
)\longrightarrow \mathcal{PAA}(\mathbb{R},\mathbb{R},\mu )$ is indeed a
 strict contraction. Therefore $\Gamma$ has a unique fixed point in  $\mathcal{PAA}(%
\mathbb{R},\mathbb{R},\mu )$ and equation \eqref{X10} has a unique measure paa solution.
\end{proof}
\section{Applications}
Let a
measure $\mu$
be defined by
$d\mu(t)=\rho (t)dt,$
where
$\rho (t)=\exp(\sin t),\,t\in \mathbb{R}.$
Then $\mu\in \mathcal{M}$
satisfies hypothesis
\textrm{{($M_1$)}.}
Since $2+\sin t\geq \sin(-t),$
 it follows that  if
 $I=[a,b]$,
 we have
 $1 +e^2\mu(I)\geq \mu(-I)$
 and so hypothesis
  {($M_2$)} is also satisfied.

 Consider the following integro-differential equations with reflection and delay.
\begin{eqnarray}\label{ref2}
   x'(t) &=& \sqrt{2}x(t)+x(-t)+\frac{\exp(-|t|)}{9}[\sin x(t-p) +\cos x(-t+p)]\nonumber\\
   &+&\int_{t}^{+\infty
}K(s-t)\frac{\exp(-|s|)}{9}[\sin x(s-p) +\cos x(-s+p)]ds \\
  &+& \int_{-t}^{+\infty }K(s+t)\frac{\exp(-|s|)}{9}[\sin x(s-p) +\cos x(-s+p)]ds,\nonumber
\end{eqnarray}
where $K(s)=\exp(-s)$, for all $s\in \mathbb{R}^+$ and  $p$ is a strictly positive real number
which denotes the delay. If we put $\beta(t)=t-p$, then hypothesis {($M_0$)} is satisfied, cf. Ben-Salah et al.\cite{mounir}.
Then
equation \eqref{ref2} is a
special case of  equation \eqref{X10} if we take
$$a=\sqrt{2}, b=1, \lambda=\sqrt{a^2-b^2}=1 \ \mbox{ and} \ f(t,x,y)=h(t,x,y)=\frac{\exp(-|t|)}{9}[\sin x +\cos y ].$$
Let $p=q=\frac{1}{2}$. Then
\begin{equation*}
|f(t,x_1,y_1)-f(t,x_2,y_2)|\leq L_f(t)\Big(|x_1-x_2|+|y_1-y_2|\Big), \ \mbox{for all} \  (x_1,y_1),\; (x_2,y_2)\in \mathbb{R}^2,
\end{equation*}
and
\begin{equation*}
|h(t,x_1,y_1)-h(t,x_2,y_2)|\leq L_h(t)\Big(|x_1-x_2|+|y_1-y_2|\Big), \ \mbox{for all} \  (x_1,y_1),\; (x_2,y_2)\in \mathbb{R}^2,
\end{equation*}
where $$[t\rightarrow L_f(t)=L_h(t)=\frac{\exp(-|t|)}{9} ]\in \mathcal{L}^2(\mathbb{R},\mathbb{R},dx)\cap
\mathcal{L}^2(\mathbb{R},\mathbb{R},d\mu),$$
 since
 $$\|L_f\|_{\mathcal{L}^2(\mathbb{R},\mathbb{R},dx)}=\|L_h\|_{\mathcal{L}^2(\mathbb{R},\mathbb{R},dx)}=\frac{1}{9}\ \  \mbox{ and}  \ \
 \|L_f\|_{\mathcal{L}^2(\mathbb{R},\mathbb{R},d\mu)}= \|L_h\|_{\mathcal{L}^2(\mathbb{R},\mathbb{R},d\mu)}\leq\frac{1}{9}\sqrt{e}.$$ This implies that
 hypothesis {($h_3$)} is satisfied. Since
$$\|L_f\|_{\mathcal{L}^2(\mathbb{R},\mathbb{R},dx)}
+2\Big(\int_0^{+\infty}(K(y))^2dy\Big)^{\frac{1}{2}}\|L_h\|_{\mathcal{L}^2(\mathbb{R},\mathbb{R},dx)}$$ $$=\frac{\sqrt{2}+1}{9}<\displaystyle
\frac{\lambda \sqrt{q\lambda} }{|\lambda-a|+|\lambda+a|+2|b| }=\frac{1}{\sqrt{2}+2},$$
we
can
deduce
that
all assumptions of Theorem \ref{final2} are satisfied and   thus
equation \eqref{ref2} has a unique
$\mu$-paa solution.
\section{Epilogue}
 In practice, the purely periodic phenomena is negligible, which gives the idea to find other solutions and consider single measure paa oscillations.
  Based on composition, completeness, Banach fixed point theorem, and change of variables theorems, we proved two very important results concerning the existence and uniqueness of a single measure paa solution of a new scalar integro-differential system.
Compared to previous works, this is first study of oscillations and dynamics of
single measure  paa solutions for certain integro-differential
equations with reflection for the case when $\beta(t)\neq t$.
Miraoui~\cite{miraoui2020} studied pap solutions with two measures for our equation \eqref{X10} for the case when $K=0$ or $h=0$ and  $\beta(t)=t$.
Ait Dads et al. \cite{ait2020} described equation \eqref{X10} with matrix coefficients for the case when $\beta(t)=t$.
On the other hand, we studied the impact  of functions $K, f,h$ and $\beta$ on the uniqueness of the single measure paa solutions for equation \eqref{X10}.
Note that
 in the special case
 when $\beta(t)=t$,  hypotheses {($M_0$)} and {($h_0$)} are satisfied, therefore the following new results can be deduced
 from  Theorems 1.1 and
1.2.
\begin{corollary}\label{th1}
Suppose that $f,h \in \mathcal{PAA}(\mathbb{R},\mathbb{R},\mu)$ and
that
hypotheses {($h_1$)--($h_4$)}
and
{($M_1$)-($M_2$)}
hold. Then the following equation
\begin{eqnarray}\label{X11}
  u^{\prime }(y) &=& au(y)+bu(-y)+f(y,u(y),u(-y))+\int_{y}^{+\infty
}K(s-y)h(s,u(s),u(-s))ds \nonumber\\
  &+& \int_{-y}^{+\infty }K(s+y)h(s,u(s),u(-s))ds,\; y\in
\mathbb{R},
\end{eqnarray}
has a unique $\mu$-paa solution if and only if
\begin{equation*}
 \displaystyle\frac{|\lambda-a|+|\lambda+a|+2|b|}{\lambda^2}(L_f+2cL_h)
< 1.
\end{equation*}
\end{corollary}
\begin{corollary}\label{th12}
Suppose that $f  \in \mathcal{PAA}(\mathbb{R},\mathbb{R},\mu)$ and that hypotheses {($h_1$)--($h_2$)}
and
 {($M_1$)-($M_2$)}
hold. Then the following equation
\begin{eqnarray}\label{X11}
  u^{\prime }(y) &=& au(y)+bu(-y)+f(y,u(y),u(-y)),\; y\in
\mathbb{R},
\end{eqnarray}
has a unique $\mu$-paa solution if and only if
\begin{equation*}
 \displaystyle\frac{|\lambda-a|+|\lambda+a|+2|b|}{\lambda^2}L_f
< 1.
\end{equation*}
\end{corollary}
\begin{corollary}
\label{final21} Suppose that $f,h \in PAA(\mathbb{R}\times\mathbb{R}
^{2},\mathbb{R},\mu)$ and that
 hypotheses
{($h_1$),  ($h'_2$)--($h'_4$)}
and
{($M_1$)-($M_2$)} hold. Then equation \eqref{X11} has a unique
$\mu$-paa solution
 if and only if
$$\|L_f\|_{\mathcal{L}^p(\mathbb{R},\mathbb{R},dx)}+2(\int_0^{+\infty}(K(y))^q)^{\frac{1}{q}}\|L_h\|_{\mathcal{L}^p(\mathbb{R},\mathbb{R},dx)}<\displaystyle%
\frac{\lambda (q\lambda)^{\frac{1}{q}} }{|\lambda-a|+|\lambda+a|+2|b| }.$$
\end{corollary}
\begin{corollary}
\label{final22} Suppose that $f\in PAA(\mathbb{R}\times\mathbb{R}
^{2},\mathbb{R},\mu)$ and
that  hypotheses
{($h_1$),  ($h'_2$)},
and
{($M_1$)-($M_2$)} hold. Then equation \eqref{X3} has a unique
$\mu$-paa solution
 if and only if
$$\|L_f\|_{\mathcal{L}^p(\mathbb{R},\mathbb{R},dx)}<\displaystyle
\frac{\lambda (q\lambda)^{\frac{1}{q}} }{|\lambda-a|+|\lambda+a|+2|b| }.$$
\end{corollary}

\bigskip
{\bf Acknowledgements.}
This research was supported by the Slovenian Research Agency grants P1-0292, N1-0114, N1-0083, N1-0064, and J1-8131.


\begin{thebibliography}{99}
\bibitem{Ad} M. Adivar, H.C. Koyuncuo\v{g}lu, Almost automorphic solutions of discrete delayed neutral system,
 Journal of Mathematical Analysis and Applications 435 (1) (2016) 532-550.
\bibitem{2} A.R. Aftabizadeh, Y.K. Huang, Bounded
solutions for differential equations with reflection of the
argument, Journal of Mathematical Analysis and Applications
135 (1988) 31-37.
\bibitem{1} A.R. Aftabizadeh, J. Wiener, Boundary value
problems for differential equations with reflection of argument,
International Journal of Mathematics and Mathematical Sciences 8
(1985) 151-163.
\bibitem{A1} E. Ait Dads, K. Ezzinbi, M. Miraoui, $(\mu,\nu)$-Pseudo almost
automorphic solutions for some nonautonomous differential
equations, Int. J. Math. 26 (2015) 1-21.
\bibitem{Ait} E. Ait Dads, S. Fatajou, L. Lhachimi, Pseudo
almost automorphic solutions for differential equations involving
reflection of the argument, International Scholarly Research
Network ISRN Mathematical Analysis Volume 2012, 21 pp.
\bibitem{ait2020} E. Ait Dads, S. Khelifi, M. Miraoui, On the integro-differential equations with reflection. Mathematical Methods in the Applied Sciences 43 (17) (2020)  10262-10275.
\bibitem{ba} A. Baskakov, V. Obukhovskii, P. Zecca, Almost periodic solutions at infinity of differential equations and inclusions,
Journal of Mathematical Analysis and Applications
462 (1) (2018) 747-763.
\bibitem{mounir} M. Ben-Salah, M. Miraoui, A. Rebey, New results for some neutral partial functional differential equations, Results
    in Mathematics (2019) 74:181, https://doi.org/10.1007/s00025-019-1106-8 (2019).
\bibitem{Khalil2} J. Blot, P. Cieutat, K. Ezzinbi, Measure
theory and pseudo almost automorphic functions: New developments and
applications, Nonlinear Analysis 75 (2012) 2426-2447.
\bibitem{Boch} S. Bochner,
{Continuous mappings of almost automorphic and almost periodic functions,} Proceedings of the
    National Academy of Sciences of the United States of America 52 (1964) 907-910.
\bibitem{ref4} F. Ch\'{e}rif, M. Miraoui,  New results for a Lasota-Wazewska model, International Journal of Biomathematics 12 (2)
    (2019) 1950019.
\bibitem{Diagana} T. Diagana, Pseudo almost periodic solutions to some differential equations, Nonlinear Anal. 60 (2005), 1277-1286.
\bibitem{D11} T. Diagana, K. Ezzinbi, M. Miraoui, Pseudo-almost periodic and
pseudo-almost automorphic solutions to some evolution equations
involving theoretical measure theory, Cubo 16 (2) (2014) 1-31.
\bibitem{8} C.P. Gupta, Existence and uniqueness theorem for
boundary value problems involving reflection of the argument,
Nonlinear Analysis TAM 11 (1987) 1075-1083.
\bibitem{9} C.P. Gupta, Two point boundary value problems
involving reflection of the argument,  International Journal of
Mathematics and Mathematical Sciences 10 (1987) 361-371.
\bibitem{KN} F. Kong, J.J. Nieto, Almost periodic dynamical behaviors of the hematopoiesis model with mixed discontinuous harvesting terms. Discrete \& Continuous Dynamical Systems - B 24 (11) (2019), 5803-5830.
\bibitem{li} K.X. Li, Weighted pseudo almost automorphic solutions for nonautonomous SPDEs driven by Levy noise, Journal of
    Mathematical Analysis and Applications 427 (2015) 686-721.
\bibitem{miraoui2020} M. Miraoui, Measure pseudo almost periodic solutions for
differential equations with reflection, Appl. Anal. (2020), doi: 10.1080/00036811.2020.1766026.
\bibitem{mir3} M. Miraoui, Existence of $\mu$-pseudo almost periodic solutions to some evolution equations, Mathematical Methods in
    the Applied Sciences 40 (13) (2017) 4716-4726.
\bibitem{mir2} M. Miraoui, $\mu$-Pseudo almost automorphic
solutions for some differential equations with reflection of the argument,
Numerical Functional Analysis and Optimization 38 (3) (2017) 376-394.
\bibitem {mir4} M. Miraoui, K. Ezzinbi, A. Rebey,
$\mu$-Pseudo Almost periodic solutions in $\alpha$-norm to some neutral partial differential equations with finite delay, Dynamics of
Continuous Discrete and Impulsive Systems, Canada  (2017) 83-96.
\bibitem{ref5} M. Miraoui, N. Yaakobi, Measure pseudo almost periodic solutions of shunting inhibitory cellular neural networks with
    mixed delays, Numerical Functional Analysis and Optimization 40 (5) (2019) 571-585.
\bibitem{g} G.M. N'Gu\'{e}r\'{e}kata, Almost Automorphic and
Almost Periodic Functions in Abstract Spaces, Kluwer Academic
Plenum Publishers, New York, Boston, Moscow, London, 2001.
\bibitem{D2}  N.S. Papageorgiou, V.D. R\u adulescu, D.D. Repov\v{s},
Periodic solutions for a class of evolution inclusions, Computers and Mathematics with Applications 75 (2018) 3047-3065.
\bibitem{D1} N.S. Papageorgiou, V.D. R\u adulescu, D.D. Repov\v{s},
Periodic solutions for implicit evolution inclusions, Evolution Equation and Control Theory 8 (3) (2019) 621-631.
\bibitem{11} D. Piao, Periodic and almost periodic solutions
for differential equations with reflection of the argument,
Nonlinear Analysis TAM 57 (2004) 633-637.
\bibitem{12} D. Piao,  Pseudo almost periodic solutions for
differential equations involving reflection of the argument, J.
Korean Math. Soc. 41 (4) (2004) 747-754.
\bibitem{sc} A.N. Sharkovskii, Functional-Differential Equations with a Finite Group of Argument Transformations in Asymptotic Behavior of Solutions of Functional-Differential Equations, Akad. Nauk Ukrain., Inst. Math., Kiev, 1978, 118-142.
\bibitem{X} N. Xin, D. Piao, Weighted pseudo almost
periodic solutions for differential equations involving reflection
of the argument, International Journal of Physical Sciences
7 (11) (2012) 1806-1810.
\bibitem{14} C. Zhang, Pseudo almost periodic solutions of
some differential equations, Journal  of Mathematical Analysis and
Applications 181 (1994) 62-76.
\end{thebibliography}
\end{document}